\documentclass[12pt]{amsart}
\usepackage{mathrsfs}

\usepackage{amssymb,amsfonts,amsthm,amsmath}
\usepackage{epsfig}
\usepackage[utf8]{inputenc}
\usepackage{mathrsfs}
\usepackage{lscape}
\usepackage{amssymb,amsmath,graphicx,color,textcomp, amsthm,bbm,bbold, enumerate,booktabs}
\usepackage{chngcntr}
\usepackage{apptools}
\AtAppendix{\counterwithin{theorem}{section}}
\definecolor{Red}{cmyk}{0,1,1,0}

\definecolor{verde}{cmyk}{1,0,1,0}

\definecolor{loka}{cmyk}{.5,0,1,.5}
\definecolor{azul}{cmyk}{1,1,0,0}

\numberwithin{equation}{section}

\newcommand{\be}{\begin{equation}}
\newcommand{\ee}{\end{equation}}

\newtheorem{definition}{Definition}

\newtheorem{corolario}{Corollary}

\newtheorem{teorema}{Theorem}

\begin{document}
\title{Ulam-Hyers-Rassias stability for a class of fractional integro-differential equations}
\author{J. Vanterler da C. Sousa$^1$}
\address{$^1$ Department of Applied Mathematics, Institute of Mathematics,
 Statistics and Scientific Computation, University of Campinas --
UNICAMP, rua S\'ergio Buarque de Holanda 651,
13083--859, Campinas SP, Brazil\newline
e-mail: {\itshape \texttt{vanterlermatematico@hotmail.com, capelas@ime.unicamp.br }}}
\author{E. Capelas de Oliveira$^1$}

\begin{abstract} By means of the recent $\psi$-Hilfer fractional derivative and of the Banach fixed-point theorem, we investigate stabilities of Ulam-Hyers, Ulam-Hyers-Rassias and semi-Ulam-Hyers-Rassias on closed intervals $[a,b]$ and $[a,\infty)$ for a particular class of fractional integro-differential equations.

\vskip.5cm
\noindent
\emph{Keywords}:  Fractional integro-differential equations, Ulam-Hyers stability, Ulam-Hyers-Rassias stability, semi-Ulam-Hyers-Rassias stability, Banach fixed-point theorem, $\psi-$Hilfer fractional derivative.
\newline 
MSC 2010 subject classifications. 26A33; 34A08, 34K20, 37C25 .
\end{abstract}
\maketitle

\section{Introduction}
In recent years, the fractional calculus (FC) has received increasing attention in the scientific community and plays a crucial role in several areas, of which we mention: pure and applied mathematics, physics, engineering and medicine, among others, specifically for their importance in theory and applications \cite{RHM,IP,KSTJ}. In fact, it is important to note that, the large number of definitions of fractional derivatives and integrals that emerged during this period reveals that the researchers proposed to work in this area of knowledge because it continues to expand \cite{JERO,SAMKO,KSTJ,ZE1,ZE2}.

On the other hand, there was considerable growth in the study of the existence and uniqueness of solutions of fractional differential equations and integral equations \cite{exis2,exis5,exis6,exis7,exis8,exis9,exis10,exis11}. In fact, it is possible to note that researches involving the existence of solutions of linear and nonlinear fractional differential equations are object of recent studies \cite{exis1,exis3,exis4}.

The investigation of Ulam-Hyers stability and Ulam-Hyers-Rassias stability became the object of research by several mathematicians and the study of this area has become one of the central themes of mathematical analysis. Afterwards, with the wide expansion of the FC, the study of stability by means of fractional derivatives, has attracted the attention of the specialized community \cite{est3,est4,est5,est6,est8,est9}. Recently, using the $\psi$-Hilfer fractional derivative and the Banach fixed-point theorem, a considerable attention was paid to the study of the stability of Ulam-Hyers and Ulam-Hyers-Rassias involving a nonlinear differential integral equation \cite{est1,est2,est42}. 

In addition, Wang et al. \cite{exis11} discussed stability of fractional differential equations in several works and obtained some new and interesting stability results. Unfortunately, it is important to mention that there are few studies related to the stability of fractional equations of the impulsive type, among which we cite \cite{imp1,imp2,imp3,imp4,imp5,imp6}. In this sense, we are working on a research project in which, is directing to the study of Ulam-Hyers stability of impulsive fractional differential equations. Also recent is the book by Benchohra et al. \cite{exis10} where they address the existence and stability of solutions to problems of initial value and boundary problems for differential and integral functional equations, using the Riemann-Liouville and Caputo fractional derivatives

In this paper, we consider the following class of fractional integro-differential equations
\begin{equation}\label{eq1}
\left\{ 
\begin{array}{ccc}
^{H}\mathbb{D}_{a+}^{\alpha ,\beta ;\psi }y\left( x\right) & = & \displaystyle f\left(
x,y\left( x\right) ,\int_{a}^{x}K\left( x,\tau ,y\left( \tau \right)
,y\left( \delta \left( \tau \right) \right) \right) d\tau \right) \\ 
I_{a+}^{1-\gamma ;\psi }y\left( a\right) & = & c,\text{ }\gamma =\alpha
+\beta \left( 1-\alpha \right) \text{\ \ }
\end{array}
\right.
\end{equation}
where $^{H}\mathbb{D}_{a+}^{\alpha ,\beta ;\psi }(\cdot)$ is $\psi$-Hilfer fractional derivative \cite{ZE1}, $I_{a+}^{1-\gamma ;\psi }(\cdot)$ is $\psi$-Riemann-Liouville fractional integral \cite{ZE1}, with $0<\alpha <1$, $0\leq \beta \leq 1$, $y\in C^{1}\left[ a,b\right] $, for $x\in \left[ a,b\right]$, $a$ and $b$ are fixed real numbers, $f:\left[ a,b\right] \times  \mathbb{C}\times \mathbb{C}\rightarrow \mathbb{C}$ and $K:[a,b]\times[a,b]\times\mathbb{C}\times\mathbb{C}\rightarrow\mathbb{C}$ are continuous functions, and $\delta :\left[ a,b\right] \rightarrow \left[ a,b\right] $ is a continuous delay function.

The main objective of the paper is to analyze the various types of stability, that is, Ulam-Hyers, Ulam-Hyers-Rassias and semi-Ulam-Hyers-Rassias in the interval $[a,b]$ for a class of fractional integro-differential equations, through the Banach fixed-point theorem. Besides that, the stability of Ulam-Hyers-Rassias in the semi infinite interval $[a,\infty)$ is also discussed. 

The paper is organized as follows: Section 2, presents, as preliminaries, the definition of the $\psi$-Hilfer fractional derivative, fractional integral of Riemann-Liouville with respect to another function and some important theorems, as well as the spaces in which such operators and theorems are defined. Besides that, we present the concepts of Ulam-Hyers, Ulam-Hyers-Rassias and semi-Ulam-Hyers-Rassias stabilities and the Banach fixed-point theorem. In Section 3, our main result, we discuss the Ulam-Hyers, Ulam-Hyers-Rassias and semi-Ulam-Hyers-Rassias stabilities in the interval $[a,b]$. Also, Ulam-Hyers-Rassias stability in the semi infinite interval $[a,\infty)$ is also discussed. Concluding remarks close the paper.


\section{Preliminaries}

In this section we introduce the function spaces in which they are of paramount importance to define the $\psi$-Riemann-Liouville fractional integral and the $\psi$-Hilfer fractional derivative. In this sense, we present two important theorems in obtaining the main results. Also, we introduce the concept of stability of Ulam-Hyers and Ulam-Hyers-Rassias and Banach fixed point theorem.

Let $\left[ a,b\right] $ $\left( 0<a<b<\infty \right) $ be a finite interval on the half-axis $\mathbb{R}^{+}$ and let $C\left[ a,b\right] $ be the space of continuous function $f$ on $\left[ a,b\right] $ with the norm \cite{ZE1}
\begin{equation}
\left\Vert f\right\Vert _{C\left[ a,b\right] }=\underset{x\in \left[ a,b%
\right] }{\max }\left\vert f\left( x\right) \right\vert .
\end{equation}

The weighted space $C_{1-\gamma ;\psi }\left[ a,b\right] $ of continuous $f$
on $\left( a,b\right] $ is defined by 
\begin{equation}
C_{1-\gamma ;\psi }\left[ a,b\right] =\left\{ f:\left( a,b\right]
\rightarrow \mathbb{R};\left( \psi \left( x\right) -\psi \left( a\right)
\right) ^{1-\gamma }f\left( x\right) \in C\left[ a,b\right] \right\}
\end{equation}%
$0\leq \gamma <1$ with the norm 
\begin{equation}
\left\Vert f\right\Vert _{C_{1-\gamma ;\psi }\left[ a,b\right] }=\left\Vert
\left( \psi \left( x\right) -\psi \left( a\right) \right) ^{1-\gamma
}f\left( x\right) \right\Vert _{C\left[ a,b\right] }=\underset{x\in \left[
a,b\right] }{\max }\left\vert \left( \psi \left( x\right) -\psi \left(
a\right) \right) ^{1-\gamma }f\left( x\right) \right\vert .
\end{equation}

The weighted space $C_{\gamma ;\psi }^{n}\left[ a,b\right] $ of continuous $f $ on $\left( a,b\right] $ is defined by 
\begin{equation}
C_{\gamma ;\psi }^{n}\left[ a,b\right] =\left\{ f:\left( a,b\right]
\rightarrow \mathbb{R};f\left( x\right) \in C^{n-1}\left[ a,b\right] ;\text{ 
}f^{\left( n\right) }\left( x\right) \in C_{\gamma ;\psi }\left[ a,b\right]
\right\}
\end{equation}
$0\leq \gamma <1$ with the norm 
\begin{equation}
\left\Vert f\right\Vert _{C_{\gamma ;\psi }^{n}\left[ a,b\right] }=\overset{
n-1}{\underset{k=0}{\sum }}\left\Vert f^{\left( k\right) }\right\Vert _{C
\left[ a,b\right] }+\left\Vert f^{\left( n\right) }\right\Vert _{C_{\gamma
;\psi }\left[ a,b\right] }.
\end{equation}

\begin{definition}{\rm\cite{ZE1}} Let $\left( a,b\right) $ $\left( -\infty \leq a<b\leq \infty \right) $ be a finite interval {\rm{(or infinite)}} of the real line $\mathbb{R}$ and let $\alpha >0$. Also let $\psi \left( x\right) $ be an increasing and positive monotone function on $\left( a,b\right] ,$ having a continuous derivative $\psi ^{\prime }\left( x\right)$ {\rm{(we denote first derivative as $\dfrac{d}{dx}\psi(x)=\psi'(x)$)}} on $\left( a,b\right) $. The left-sided fractional integral of a function $f$ with respect to a function $\psi $ on $ 
\left[ a,b\right] $ is defined by 
\begin{equation}\label{eq7}
I_{a+}^{\alpha ;\psi }f\left( x\right) =\frac{1}{\Gamma \left( \alpha
\right) }\int_{a}^{x}\psi ^{\prime }\left( s\right) \left( \psi \left(
x\right) -\psi \left( s\right) \right) ^{\alpha -1}f\left( s\right) ds.
\end{equation}

The right-sided fractional integral is defined in an analogous form.
\end{definition}

As the aim of this paper is present some types of the stabilities involving a class of fractional integrodifferential equations by means of $\psi$-Hilfer operator, we introduce this operator.

\begin{definition}{\rm\cite{ZE1}} Let $n-1<\alpha <n$ with $n\in \mathbb{N},$ let $\ I=\left[ a,b\right] $ be an interval such that $-\infty \leq a<b\leq \infty $ and let $f,\psi \in C^{n}\left[ a,b\right] $ be two functions such that $\psi $ is increasing and $\psi ^{\prime }\left( x\right) \neq 0,$ for all $x\in I$. The left-sided $\psi -$Hilfer fractional derivative $^{H}\mathbb{D}_{a+}^{\alpha ,\beta ;\psi }\left( \cdot \right) $ of a function $f$ of order $\alpha $ and type $0\leq \beta \leq 1,$ is defined by 
\begin{equation}\label{eq8}
^{H}\mathbb{D}_{a+}^{\alpha ,\beta ;\psi }f\left( x\right) =I_{a+}^{\beta \left( n-\alpha \right) ;\psi }\left( \frac{1}{\psi ^{\prime }\left( x\right) }\frac{d}{dx}\right) ^{n}I_{a+}^{\left( 1-\beta \right) \left(
n-\alpha \right) ;\psi }f\left( x\right) .
\end{equation}

The right-sided $\psi-$Hilfer fractional derivative is defined in an analogous form.
\end{definition}

\begin{teorema}\label{teo1} If $f\in C^{1}\left[ a,b\right] $, $0<\alpha <1$ and $0\leq \beta \leq 1$, then
\begin{equation*}
I_{a+}^{\alpha ;\psi }\text{ }^{H}\mathbb{D}_{a+}^{\alpha ,\beta ;\psi }f\left( x\right) =f\left( x\right) -\frac{\left( \psi \left( x\right) -\psi\left( a\right) \right) ^{\gamma -1}}{\Gamma \left( \gamma \right) }I_{a+}^{\left( 1-\beta \right) \left( 1-\alpha \right) ;\psi }f\left(a\right).
\end{equation*}
\end{teorema}

\begin{proof}
See {\rm\cite{ZE1}}.
\end{proof}

\begin{teorema}\label{teo2} If $f\in C^{1}\left[ a,b\right] $, $0<\alpha <1$ and $0\leq \beta\leq 1$, we have
\begin{equation}
^{H}\mathbb{D}_{a+}^{\alpha ,\beta ;\psi }\text{ }I_{a+}^{\alpha ;\psi }f\left( x\right) =f\left( x\right) .
\end{equation}
\end{teorema}

\begin{proof}
See {\rm\cite{ZE1}} .
\end{proof}

We present the concepts of Ulam-Hyers and Ulam-Hyers-Rassias stabilities, both fundamental in the study of the main results of the article. The following Definition \ref{def4} and Definition \ref{def5}, were adapted from paper \cite{castro}.

\begin{definition}\label{def4} For each function $y$ satisfying 
\begin{equation}\label{eq11}
\left\vert ^{H}\mathbb{D}_{a+}^{\alpha ,\beta ;\psi }y\left( x\right) -f\left( x,y\left( x\right) ,\int_{a}^{x}K\left( x,\tau ,y\left( \tau \right) ,y\left( \delta \left( \tau \right) \right) \right) d\tau \right)\right\vert \leq \theta
\end{equation}
$x\in \left[ a,b\right] ,$ where $\theta \geq 0$, there is a solution $y_{0}$ of the fractional integro-differential equation and a constant $C>0$ independent of $y$ and $y_{0}$ such that
\begin{equation}\label{eq12}
\left\vert y\left( x\right) -y_{0}\left( x\right) \right\vert \leq C\theta
\end{equation}
for all $x\in \left[ a,b\right] $, then we say that the integro-differential equation has the Ulam-Hyers stability.

If instead of $\theta $, in \textnormal{ Eq.(\ref{eq11})} and \textnormal{ Eq.(\ref{eq12})}, we have a nonnegative function $\sigma $ defined on $\left[ a,b\right] $, then we say that the fractional integro-differential equation has the Ulam-Hyers-Rassias stability.
\end{definition}

\begin{definition} \label{def5} If for each function $y$ satisfying
\begin{equation}\label{eq13}
\left\vert ^{H}\mathbb{D}_{a+}^{\alpha ,\beta ;\psi }y\left( x\right) -f\left( x,y\left( x\right) ,\int_{a}^{x}K\left( x,\tau ,y\left( \tau \right) ,y\left( \delta \left( \tau \right) \right) \right) d\tau \right)\right\vert \leq \theta
\end{equation}
$x\in \left[ a,b\right] $, where $\theta \geq 0$, there is a solution $y_{0}$ of the fractional integro-differential equation and a constant $C>0$ independent of $y$ and $y_{0}$ such that
\begin{equation}\label{eq14}
\left\vert y\left( x\right) -y_{0}\left( x\right) \right\vert \leq C\sigma \left( x\right)
\end{equation}
$x\in \left[ a,b\right] $ for some nonnegative function $\sigma $ defined on  $\left[ a,b\right] $, then we say that the fractional integro-differential equation has the so-called semi-Ulam-Hyers-Rassias stability.
\end{definition}

\begin{definition}{\rm\cite{fixed}} We say that $d:X\times X\rightarrow \left[ 0,\infty \right] $ is a generalized metric on $X$ if:
\begin{enumerate}
\item $d\left( x,y\right) =0$ if and only if $x=y;$
\item $d\left( x,y\right) =d\left( y,x\right) ,$ for all $x,y\in X;$
\item $d\left( x,z\right) \leq d\left( x,y\right) +d\left( y,z\right) $ for all $x,y,z\in X$.
\end{enumerate}
\end{definition}

\begin{teorema}\textnormal{(Banach)}\label{teo3} Let $\left( X,d\right) $ be a generalized complete metric space and $T:X\rightarrow X$ a strictly contractive operator with \ Lipschitz constant $L<1$. If there exists a nonnegative integer $k$ such that $d\left(T^{k+1}x,T^{k}x\right) <\infty $ for some $x\in X$, then the following three propositions hold true:

\begin{enumerate}
\item The sequence $\left( T^{n}x\right) _{n\in \mathbb{N}}$ converges to a fixed point $x^{\ast }$ of $T$;
\item $x^{\ast }$ is the unique fixed point of $T$ in $X^{\ast }=\left\{ y\in X;\text{ }d\left( T^{k}x,y\right) <\infty \right\}$;
\item If $y\in X^{\ast }$, then
\end{enumerate}
\begin{equation}\label{eq15}
d\left( y,x^{\ast }\right) \leq \frac{1}{1-L}d\left( Ty,y\right).
\end{equation}
\end{teorema}

\begin{proof}
See \cite{fixed}.
\end{proof}

\section{Main Results}

In this section it will be divided into subsections in which we present the main purpose of this paper. First, we study the stability of Ulam-Hyers, semi-Ulam-Hyers-Rassias and Ulam-Hyers-Rassias in the finite interval, by means of the Bielecki metric. In this sense, we also discuss the stability of Ulam-Hyers-Rassias stability in the infinite range, by means of the metric $d_{b}$.

\subsection{Ulam-Hyers-Rassias stability}

In this subsection, we present sufficient conditions to obtain the Ulam-Hyers-Rassias stability of the integro-differential fractional equation where $x\in\left[a,b\right]$, for some fixed real numbers $a$ and $b$.

Consider the space of continuously differentiable functions $C^{1}\left[ a,b\right] $ on $\left[ a,b\right] $ endowed with a Bielecki type metric 
\begin{equation}\label{eq16}
d\left( u,v\right) =\underset{x\in \left[ a,b\right] }{\sup }\frac{\left\vert u\left( x\right) -v\left( x\right) \right\vert }{\sigma \left( x\right) }
\end{equation}
where $\sigma $ is a non decreasing continuous function $\sigma :\left[ a,b\right] \rightarrow \left( 0,\infty \right) $. We recall that $\left( C^{1}\left[ a,b\right] ,d\right) $ is a complete metric space \cite{cuada}.

\begin{teorema}\label{teo4} Let $\delta :\left[ a,b\right] \rightarrow \left[ a,b\right] $ be a continuous delay function with $\delta \left( t\right) \leq t$ for all $t\in \left[ a,b\right] $ and $\sigma :\left[ a,b\right] \rightarrow \left(0,\infty \right) $ a non decreasing continuous function. In addition, suppose that there is $\xi \in \left[ 0,1\right)$, such that
\begin{equation}\label{eq17}
\frac{1}{\Gamma \left( \alpha \right) }\int_{a}^{x}\psi ^{\prime }\left( \tau \right) \left( \psi \left( x\right) -\psi \left( \tau \right) \right) ^{\alpha -1}\sigma \left( \tau \right) d\tau \leq \xi \sigma \left( x\right)
\end{equation}
for all $x\in \left[ a,b\right] .$ Moreover, suppose that $f:\left[ a,b\right] \times  \mathbb{C} \times \mathbb{C}\rightarrow \mathbb{C} $ is a continuous function satisfying the Lipschitz condition
\begin{equation}\label{eq18}
\left\vert f\left( x,u,g\right) -f\left( x,v,h\right) \right\vert \leq M\left( \left\vert u-v\right\vert +\left\vert g-h\right\vert \right)
\end{equation}
with $M>0$ and the kernel $K:\left[ a,b\right] \times \left[ a,b\right] \times\mathbb{C}\times \mathbb{C}\rightarrow \mathbb{C}$ is a continuous function satisfying the Lipschitz condition
\begin{equation}\label{eq19}
\left\vert K\left( x,u,g\right) -K\left( x,v,h\right) \right\vert \leq L\left\vert g-h\right\vert
\end{equation}
with $L>0$.

If $y\in C^{1}\left[ a,b\right] $ is such that
\begin{equation}\label{eq20}
\left\vert ^{H}\mathbb{D}_{a+}^{\alpha ,\beta ;\psi }y\left( x\right) -f\left( x,y\left( x\right) ,\int_{a}^{x}K\left( x,\tau ,y\left( \tau \right) ,y\left( \delta \left( \tau \right) \right) \right) d\tau \right)
\right\vert \leq \sigma \left( x\right) 
\end{equation}
$x\in \left[ a,b\right] ,$ and $M\left( \xi +L\xi ^{2}\right) <1$, then there is a unique function, $y_{0}\in C^{1}\left[ a,b\right] $ such that
\begin{equation*}
^{H}\mathbb{D}_{a+}^{\alpha ,\beta ;\psi }y_{0}\left( x\right) =f\left( x,y_{0}\left( x\right) ,\int_{a}^{x}K\left( x,\tau ,y_{0}\left( \tau \right) ,y_{0}\left( \delta \left( \tau \right) \right) \right) d\tau \right)
\end{equation*}
and
\begin{equation}\label{eq21}
\left\vert y\left( x\right) -y_{0}\left( x\right) \right\vert \leq \frac{\xi \sigma \left( x\right) }{1-M\left( \xi +L\xi ^{2}\right) }
\end{equation}
for all $x\in \left[ a,b\right]$.

This means that under above conditions, the fractional integro-differential equation \textnormal{Eq.(\ref{eq1})} has the Ulam-Hyers-Rassias stability.
\end{teorema}

\begin{proof}Applying the fractional integral operator $I_{a+}^{\alpha ;\psi }\left( \cdot \right) $ on both sides of the fractional equation \textnormal{Eq.(\ref{eq1})} and using Theorem \ref{teo1}, we can write
\begin{equation}
y\left( x\right) =\frac{\left( \psi \left( x\right) -\psi \left( a\right) \right) ^{\gamma -1}}{\Gamma \left( \gamma \right) }c+I_{a+}^{\alpha ;\psi }f\left( x,y\left( x\right) ,\int_{a}^{x}K\left( x,\tau ,y\left( \tau \right) ,y\left( \delta \left( \tau \right) \right) \right) d\tau \right) . \label{eq22}
\end{equation}

On the other hand, if $y$ satisfies \textnormal{Eq.(\ref{eq22})}, then $y$ satisfies \textnormal{Eq.(\ref{eq1})}.
However, applying the fractional derivative $^{H}\mathbb{D}_{a+}^{\alpha ,\beta ;\psi }\left( \cdot \right) $ on both sides of \textnormal{Eq.(\ref{eq22})}, we have
\begin{eqnarray*}
^{H}\mathbb{D}_{a+}^{\alpha ,\beta ;\psi }y\left( x\right)  &=&\text{ }^{H}\mathbb{D}_{a+}^{\alpha ,\beta ;\psi }\left[ \frac{\left( \psi \left(x\right) -\psi \left( a\right) \right) ^{\gamma -1}}{\Gamma \left( \gamma
\right) }c\right]  \\
&&+\text{ }^{H}\mathbb{D}_{a+}^{\alpha ,\beta ;\psi }I_{a+}^{\alpha ;\psi}f\left( x,y\left( x\right) ,\int_{a}^{x}K\left( x,\tau ,y\left( \tau\right) ,y\left( \delta \left( \tau \right) \right) \right) d\tau \right) .
\end{eqnarray*}

Using Theorem 2 and the expression
\begin{equation*}
^{H}\mathbb{D}_{a+}^{\alpha ,\beta ;\psi }\left( \psi \left( x\right) -\psi \left( a\right) \right) ^{\gamma -1}=0,\text{ }0<\gamma <1,
\end{equation*}
we conclude that, $y\left( x\right) $ satisfies initial value problem \textnormal{Eq.(\ref{eq1})} if, and only if, $y\left( x\right) $ satisfies the integral equation
\begin{equation}\label{a21}
y\left( x\right) =\frac{\left( \psi \left( x\right) -\psi \left( a\right) \right) ^{\gamma -1}}{\Gamma \left( \gamma \right) }c+I_{a+}^{\alpha ;\psi }f\left( x,y\left( x\right) ,\int_{a}^{x}K\left( x,\tau ,y\left( \tau \right) ,y\left( \delta \left( \tau \right) \right) \right) d\tau \right).
\end{equation}

So, consider the operator $T:C^{1}\left[ a,b\right] \rightarrow C\left[ a,b\right]$, defined by
\begin{equation}\label{a22}
Tu\left( x\right) =\frac{\left( \psi \left( x\right) -\psi \left( a\right) \right) ^{\gamma -1}}{\Gamma \left( \gamma \right) }c+I_{a+}^{\alpha ;\psi }f\left( x,u\left( x\right) ,\int_{a}^{x}K\left( x,\tau ,u\left( \tau\right) ,u\left( \delta \left( \tau \right) \right) \right) d\tau \right) 
\end{equation}
for all $x\in \left[ a,b\right] $ and $u\in C^{1}\left[ a,b\right]$.

Note that for any continuous function $u,Tu$ is also continuous. In fact,
\begin{eqnarray}
&&\left\vert Tu\left( x\right) -Tu\left( x_{0}\right) \right\vert \label{a23} \\
&=&\left\vert 
\begin{array}{c}
\dfrac{\left( \psi \left( x\right) -\psi \left( a\right) \right) ^{\gamma -1}}{\Gamma \left( \gamma \right) }c+I_{a+,x}^{\alpha ;\psi }f\left( x,u\left(x\right) ,\displaystyle\int_{a}^{x}K\left( x,\tau ,u\left( \tau \right)
,u\left( \delta \left( \tau \right) \right) \right) d\tau \right)  \\ -\dfrac{\left( \psi \left( x_{0}\right) -\psi \left( a\right) \right)^{\gamma -1}}{\Gamma \left( \gamma \right) }c-I_{a+,x_{0}}^{\alpha ;\psi}f\left( x_{0},u\left( x_{0}\right) ,\displaystyle\int_{a}^{x_{0}}K\left( x_{0},\tau,u\left( \tau \right) ,u\left( \delta \left( \tau \right) \right) \right)
d\tau \right) 
\end{array}%
\right\vert \rightarrow 0  \notag
\end{eqnarray}
when $x\rightarrow x_{0}$.

We will deduce that the operator $T$ is strictly contractive with respect to the metric Eq.(\ref{eq16}). Indeed, for all, $u,v\in C^{1}\left[ a,b\right] $ we get
\begin{eqnarray}
&&d\left( Tu,Tv\right)   \notag  \label{eq23} \\
&=&\underset{x\in \left[ a,b\right] }{\sup }\frac{\left\vert 
\begin{array}{c}
I_{a+}^{\alpha ;\psi }f\left( x,u\left( x\right) ,\displaystyle%
\int_{a}^{x}K\left( x,\tau ,u\left( \tau \right) ,u\left( \delta \left( \tau
\right) \right) \right) d\tau \right)  \\ 
-I_{a+}^{\alpha ;\psi }f\left( x,v\left( x\right) ,\displaystyle%
\int_{a}^{x}K\left( x,\tau ,v\left( \tau \right) ,v\left( \delta \left( \tau
\right) \right) \right) d\tau \right) 
\end{array}%
\right\vert }{\sigma \left( x\right) }  \notag \\
&\leq &\underset{x\in \left[ a,b\right] }{\sup }\frac{\left\vert
I_{a+}^{\alpha ;\psi }\left( M\left[ \left\vert u\left( x\right) -v\left(
x\right) \right\vert +\displaystyle\int_{a}^{x}\left\vert 
\begin{array}{c}
K\left( x,\tau ,u\left( \tau \right) ,u\left( \delta \left( \tau \right)
\right) \right)  \\ 
-K\left( x,\tau ,v\left( \tau \right) ,v\left( \delta \left( \tau \right)
\right) \right) 
\end{array}%
\right\vert d\tau \right] \right) \right\vert }{\sigma \left( x\right) } 
\notag \\
&\leq &M\underset{x\in \left[ a,b\right] }{\sup }\frac{I_{a+}^{\alpha ;\psi
}\left( \left\vert u\left( x\right) -v\left( x\right) \right\vert \right) }{
\sigma \left( x\right) }+ML\underset{x\in \left[ a,b\right] }{\sup }\frac{
I_{a+}^{\alpha ;\psi }\left( \displaystyle\int_{a}^{x}\left\vert u\left(
\delta \left( \tau \right) \right) -v\left( \delta \left( \tau \right)
\right) \right\vert d\tau \right) }{\sigma \left( x\right) }  \notag \\
&\leq &M\xi\underset{x\in \left[ a,b\right] }{\sup }\frac{\left\vert u\left(
x\right) -v\left( x\right) \right\vert }{\sigma \left( x\right) }+ML\underset{\tau \in \left[ a,b\right] }{\sup }\frac{\left\vert u\left( \delta \left( \tau \right) \right) -v\left(
\delta \left( \tau \right) \right) \right\vert d\tau }{\sigma \left( \tau
\right) }  \notag \\
&&+\underset{x\in \left[ a,b\right] }{\sup }\frac{I_{a+}^{\alpha ;\psi }%
\left[ \displaystyle\int_{a}^{x}\sigma \left( \tau \right) d\tau \right] }{%
\sigma \left( x\right) }  \notag \\
&\leq &M\xi\text{ }d\left( u,v\right) +ML\text{ }d\left( u,v\right) \text{ }\underset{x\in \left[ a,b\right] }{\sup }\frac{%
I_{a+}^{\alpha ;\psi }\left[ \xi \sigma \left( x\right) \right] }{\sigma
\left( x\right) }  \notag \\
&\leq &M\left( \xi +L\xi ^{2}\right) d\left( u,v\right) .
\end{eqnarray}

As $M\left( \xi +L\xi ^{2}\right) <1$ it follows that $T$ is strictly contractive. In this way, we can apply the above mentioned Banach fixed-point theorem which allows us to ensure that we have the Ulam-Hyers-Rassias stability for the fractional integro-differential equation.

In fact, from Eq.(\ref{eq20}), using fractional integral Eq.(\ref{eq7}) and Theorem \ref{teo1}, we have
\begin{eqnarray}\label{a26}
&&\left\vert y\left( x\right) -\frac{\left( \psi \left( x\right) -\psi
\left( a\right) \right) ^{\gamma -1}}{\Gamma \left( \gamma \right) }%
c-I_{a+}^{\alpha ;\psi }f\left( x,y\left( s\right) ,\int_{a}^{x}K\left(
x,\tau ,y\left( \tau \right) ,y\left( \delta \left( \tau \right) \right)
\right) d\tau \right) \right\vert   \notag  \label{a26} \\
&\leq &I_{a+}^{\alpha ;\psi }\sigma \left( \tau \right) 
\end{eqnarray}
$x\in \left[ a,b\right]$.

Therefore, having in mind Eq.(\ref{a22}) and Eq.(\ref{eq17}), we conclude that
\begin{equation}\label{eq24}
\left\vert y\left( x\right) -Ty\left( x\right) \right\vert \leq I_{a+}^{\alpha ;\psi }\sigma \left( \tau \right) \leq \xi \sigma \left( x\right) ,\text{ }x\in \left[ a,b\right].
\end{equation}

Note that, 
\begin{equation*}
-M\left( \xi +L\xi ^{2}\right) \leq -L\Rightarrow 1-M\left( \xi +L\xi ^{2}\right) \leq 1-L\Rightarrow \frac{1}{1-L}\leq \frac{1}{1-M\left( \xi +L\xi ^{2}\right) }.
\end{equation*}

Consequently, Eq.(\ref{eq21}) follows directly from the definition of the metric $d$, Eq.(\ref{eq15}) and Eq.(\ref{eq24}).
\end{proof}

\subsection{Semi-Ulam-Hyers-Rassias and Ulam-Hyers stabilities}

The next step, presented as a theorem contains the proof of the following result: the fractional integro-differential equation admits semi-Ulam-Hyers-Rassias and Ulam-Hyers stabilities.

\begin{teorema}\label{teo5} Let $\delta :\left[ a,b\right] \rightarrow \left[ a,b\right] $ be continuous delay function with $\delta \left( t\right) \leq t$ for all $t\in \left[ a,b\right] $ and $\sigma :\left[ a,b\right] \rightarrow \left(
0,\infty \right) $ a non-decreasing continuous function. In addition, suppose that there is $\xi \in \left[ 0,1\right) $ such that
\begin{equation}\label{eq25}
\frac{1}{\Gamma \left( \alpha \right) }\int_{a}^{x}\psi ^{\prime }\left( \tau \right) \left( \psi \left( x\right) -\psi \left( \tau \right) \right)^{\alpha -1}\sigma \left( \tau \right) d\tau \leq \xi \sigma \left( x\right),
\end{equation}
for all $x\in \left[ a,b\right] .$ Moreover, suppose that $f:\left[ a,b\right] \times \mathbb{C}\times \mathbb{C}\rightarrow \mathbb{C} $ is a continuous function satisfying the Lipschitz condition
\begin{equation}\label{eq26}
\left\vert f\left( x,u,g\right) -f\left( x,v,h\right) \right\vert \leq M\left( \left\vert u-v\right\vert +\left\vert g-h\right\vert \right) 
\end{equation}
with $M>0$ and $K:\left[ a,b\right] \times \left[ a,b\right] \times \mathbb{C} \times \mathbb{C} \rightarrow \mathbb{C}$ is a continuous kernel function satisfying the Lipschitz condition
\begin{equation}\label{eq27}
\left\vert K\left( x,t,u,w\right) -K\left( x,t,v,z\right) \right\vert \leq L\left\vert w-z\right\vert 
\end{equation}
with $L>0$. If $y\in C^{1}\left[ a,b\right] $, is such that
\begin{equation}\label{eq28}
\left\vert ^{H}\mathbb{D}_{a+}^{\alpha ,\beta ;\psi }y\left( x\right)-f\left( x,y\left( x\right) ,\int_{a}^{x}K\left( x,\tau ,y\left( \tau\right) ,y\left( \delta \left( \tau \right) \right) \right) d\tau \right)
\right\vert \leq \theta 
\end{equation}
$x\in \left[ a,b\right] $ where $\theta >0$ and $M\left( \xi +L\xi^{2}\right) <1$, then there is a unique function $y_{0}\in C^{1}\left[ a,b\right] $ such that
\begin{equation}\label{eq29}
^{H}\mathbb{D}_{a+}^{\alpha ,\beta ;\psi }y_{0}\left( x\right) =f\left( x,y_{0}\left( x\right) ,\int_{a}^{x}K\left( x,t,y_{0}\left( t\right), y_{0}\left( \delta \left( \tau \right) \right) \right) d\tau \right) 
\end{equation}
and
\begin{equation}\label{eq30}
\left\vert y\left( x\right) -y_{0}\left( x\right) \right\vert \leq \frac{\left( b-a\right) \theta \text{ }\sigma \left( x\right) }{\left[ 1-M\left( \xi +L\xi ^{2}\right) \right] \sigma \left( a\right) }
\end{equation}
for all $x\in \left[ a,b\right]$.

This means that under the above conditions, the fractional integro-differential equation \textnormal{Eq.(\ref{eq1})} has the semi-Ulam-Hyers-Rassias stability.
\end{teorema}

\begin{proof}
The first part of the proof follows the same steps as in the proof of Theorem \ref{teo2}. The main purpose here is to choose a general $\sigma$ and do as previously done in the metric function $d$, using the same idea as in Eq.(\ref{a21})-Eq.(\ref{a23}), and just take it as a constant in all the remaining places.

Consider the operator $T:C^{1}\left[ a,b\right]\rightarrow C^{1}\left[ a,b\right]$, defined by
\begin{equation}
Tu\left( x\right) =\frac{\left( \psi \left( x\right) -\psi \left( a\right) \right) ^{\gamma -1}}{\Gamma \left( \gamma \right) }c+I_{a+}^{\alpha ;\psi }f\left( x,u\left( x\right) ,\int_{a}^{x}K\left( x,\tau ,u\left( x\right)
,u\left( \delta \left( \tau \right) \right) \right) d\tau \right) 
\end{equation}
for all $x\in \left[ a,b\right] $ and $u\in C^{1}\left[ a,b\right] $ and by using the same reasoning an in Eq.(\ref{a22})-Eq.(\ref{a23}) we conclude that $T$ is strictly contractive with respect to the metric Eq.(\ref{eq16}), due to the fact that $M\left( \xi +L\xi ^{2}\right) <1$. Thus, we can again apply the Banach fixed-point theorem, which guarantees us that
\begin{equation}\label{eq35}
d\left( y,y_{0}\right) \leq \frac{1}{1-M\left( \xi +L\xi ^{2}\right) }d\left( Ty,y\right).
\end{equation}

Now, by Eq.(\ref{eq28}), using fractional integral Eq.(\ref{eq7}) and Theorem 1, we get
\begin{eqnarray}
&&\left\vert y\left( x\right) -\frac{\left( \psi \left( x\right) -\psi
\left( a\right) \right) ^{\gamma -1}}{\Gamma \left( \gamma \right) }%
c-I_{a+}^{\alpha ;\psi }f\left( x,y\left( x\right) ,\int_{a}^{x}K\left(
x,\tau ,y\left( x\right) ,y\left( \delta \left( \tau \right) \right) \right)
d\tau \right) \right\vert   \notag  \label{eq36} \\
&\leq &I_{a+}^{\alpha ;\psi }\theta .
\end{eqnarray}

Therefore, having in mind Eq.(\ref{a22}) and Eq.(\ref{eq17}), we conclude that
\begin{equation}\label{eq37}
\left\vert y\left( x\right) -Ty\left( x\right) \right\vert \leq \theta I_{a+}^{\alpha ;\psi }(1)\leq \theta \frac{\left( \psi \left( b\right) -\psi \left( a\right) \right) ^{\alpha }}{\Gamma \left( \alpha +1\right) }.
\end{equation}

Using the Eq.(\ref{eq35}), and having in mind Eq.(\ref{eq7}) and that $\sigma $ is a positive non-decreasing function, it follows 
\begin{eqnarray}\label{eq38}
\underset{x\in \left[ a,b\right] }{\sup }\frac{\left\vert y\left( x\right) -y_{0}\left( x\right) \right\vert }{\sigma \left( x\right) } &\leq &\frac{1}{1-M\left( \xi +\xi ^{2}L\right) }\underset{x\in \left[ a,b\right] }{\sup }\frac{d\left( Ty,y\right) }{\sigma \left( x\right) }  \notag \\ &\leq &\frac{\theta \left( \psi \left( b\right) -\psi \left( a\right) \right) ^{\alpha }}{\left( 1-M\left( \xi +\xi ^{2}L\right) \right) \Gamma \left( \gamma +1\right) \sigma \left( a\right) }.
\end{eqnarray}

Consequently, Eq.(\ref{eq30}) follows directly from Eq.(\ref{eq38}) and this led us to the semi-Ulam-Hyers-Rassias stability of the fractional integro-differential equation under study.

Still having in mind that $\sigma $ is a positive non-decreasing function, and considering an obvious upper bound in Eq.(\ref{eq30}), we directly obtain from the last result the following Ulam-Hyers stability of the fractional integro-differential equation Eq.(\ref{eq1}).
\end{proof}

\begin{corolario} Let $\delta :\left[ a,b\right] \rightarrow \left[ a,b\right] $ be a continuous delay function with $\delta \left( t\right) \leq t$ for all $t\in \left[ a,b\right] $ and $\sigma :\left[ a,b\right] \rightarrow \left( 0,\infty \right) $ a non-decreasing continuous function. In addition,
suppose that there is $\xi \in \left[ 0,1\right) $ such that
\begin{equation}\label{eq39}
\frac{1}{\Gamma \left( \alpha \right) }\int_{a}^{x}\psi ^{\prime }\left( \tau \right) \left( \psi \left( x\right) -\psi \left( \tau \right) \right) ^{\alpha -1}\sigma \left( \tau \right) d\tau \leq \xi \sigma \left( x\right) 
\end{equation}
for all $x\in \left[ a,b\right] .$ Moreover, suppose that $f:\left[ a,b\right] \times  \mathbb{C}\times \mathbb{C}\rightarrow \mathbb{C}$ is a continuous function satisfying the Lipschitz condition
\begin{equation}\label{eq40}
\left\vert f\left( x,u,g\right) -f\left( x,v,h\right) \right\vert \leq M\left( \left\vert u-v\right\vert +\left\vert g-h\right\vert \right) 
\end{equation}
with $M>0$ and $K:\left[ a,b\right] \times \left[ a,b\right] \times \mathbb{C}\times \mathbb{C}\rightarrow \mathbb{C}$ is a continuous kernel satisfying the Lipschitz condition
\begin{equation*}
\left\vert K\left( x,t,u,w\right) -K\left( x,t,v,z\right) \right\vert \leq L\left\vert w-z\right\vert
\end{equation*}
with $L>0$.

If $y\in C^{1}\left[ a,b\right] $ is such that
\begin{equation}\label{eq41}
\left\vert ^{H}\mathbb{D}_{a+}^{\alpha ,\beta ;\psi }y\left( x\right) -f\left( x,y\left( x\right) ,\int_{a}^{x}K\left( x,\tau ,y\left( \tau \right) ,y\left( \delta \left( \tau \right) \right) \right) d\tau \right)
\right\vert \leq \theta 
\end{equation}
$x\in \left[ a,b\right] $ where $\theta >0$ and $M\left( \xi +\xi^{2}L\right) <1$, then there is a unique function $y_{0}\in C^{1}\left[ a,b\right]$; such that
\begin{equation*}
^{H}\mathbb{D}_{a+}^{\alpha ,\beta ;\psi }y\left( x\right) =f\left( x,y_{0}\left( x\right) ,\int_{a}^{x}K\left( x,\tau ,y_{0}\left( \tau \right) ,y_{0}\left( \delta \left( \tau \right) \right) \right) d\tau \right)
\end{equation*}
and
\begin{equation*}
\left\vert y\left( x\right) -y_{0}\left( x\right) \right\vert \leq \frac{\theta \left( \psi \left( b\right) -\psi \left( a\right) \right) ^{\alpha }\sigma \left( b\right) }{\left[ 1-M\left( \xi +\xi ^{2}L\right) \right]
\Gamma \left( \gamma +1\right) \sigma \left( a\right) },\text{ }x\in \left[ a,b\right].
\end{equation*}

This means that under the above conditions, the fractional integro-differential equation \textnormal{Eq.(\ref{eq1}) } has the Ulam-Hyers stability.
\end{corolario}

\subsection{Ulam-Hyers-Rassias stability in the infinite interval case}

In this section we will discuss the Ulam-Hyers-Rassias stability associated with the fractional integro-differential equation on semi infinite interval $\left[ a,\infty \right) $ for a fixed $a\in \mathbb{R}$.

Consider the fractional integro-differential equation
\begin{equation}\label{eq42}
\left\{ 
\begin{array}{ccc}
^{H}\mathbb{D}_{a+}^{\alpha ,\beta ;\psi }y\left( x\right)  & = & f\left( x,y\left( x\right), \displaystyle\int_{a}^{x}K\left( x,\tau ,y\left( \tau \right) ,y\left( \delta \left( \tau \right) \right) \right) d\tau \right)  \\ 
I_{a+}^{1-\gamma ;\psi }y\left( a\right)  & = & c,\text{ \ \ \ \ \ \ \ \ \ \
\ \ \ \ \ \ \ \ \ \ \ \ \ \ }\gamma =\alpha +\beta \left( 1-\alpha \right) 
\end{array}
\right. 
\end{equation}
with $y\in C^{1}\left[ a,\infty \right) $, $x\in \left[ a,\infty \right) $ where $a$ is a fixed real number, $f:\left[ a,\infty \right) \times  \mathbb{C}\times \mathbb{C}\rightarrow \mathbb{C}$ and $K:\left[ a,\infty \right) \times \left[ a,\infty \right) \times \mathbb{C}\times \mathbb{C} \rightarrow \mathbb{C}$ are continuous functions and $\delta :\left[ a,\infty \right) \rightarrow\left[ a,\infty \right) $ is a continuous delay function which therefore fulfills $\delta \left( \tau \right) \leq \tau $ for all $\tau \in \left[ a,\infty \right)$.

Here, we will use a recurrence procedure due to the result of the corresponding finite interval case. Let us now consider a fixed non-decreasing continuous $\sigma :\left[ a,\infty \right)\rightarrow \left( \varepsilon ,\omega \right) $ for some $\varepsilon ,\omega >0$ and the space $C_{b}^{1}\left[ a,\infty \right) $ of bounded differentiable functions endowed with the metric
\begin{equation}\label{eq43}
d_{b}\left( u,v\right) =\underset{x\in \left[ a,\infty \right) }{\sup }\frac{\left\vert u\left( x\right)-v\left( x\right) \right\vert }{\sigma \left(x\right) }.
\end{equation}

\begin{teorema}\label{teo6} Let $\delta :\left[ a,\infty \right) \rightarrow \left[ a,\infty \right) $ be a continuous delay function with $\delta \left( t\right) \leq t$, for all $t\in \left[ a,\infty \right) $, and $\sigma :\left[ a,\infty
\right) \rightarrow \left( \varepsilon ,\omega \right)$, for some $\varepsilon ,\omega >0$, a non-decreasing continuous function. Suppose that there is $\xi \in \left[ 0,1\right] $ such that
\begin{equation} \label{eq44}
\frac{1}{\Gamma \left( \alpha \right) }\int_{a}^{x}\psi ^{\prime }\left( \tau \right) \left( \psi \left( x\right) -\psi \left( \tau \right) \right) ^{\alpha -1}\sigma \left( \tau \right) d\tau \leq \xi \sigma \left( x\right)
\end{equation}
for all $x\in \left[ a,\infty \right) .$ Moreover, suppose that, $f:\left[ a,\infty \right) \times 
\mathbb{C}\times \mathbb{C}\times \rightarrow \mathbb{C}$ is a continuous function satisfying the Lipschitz condition
\begin{equation}\label{eq45}
\left\vert f\left( x,u,g\right) -f\left( x,v,h\right) \right\vert \leq M\left( \left\vert u-v\right\vert +\left\vert g-h\right\vert \right) 
\end{equation}
with $M>0$, and the kernel $K:\left[ a,\infty \right) \times \left[ a,\infty \right) \times \mathbb{C}\times \mathbb{C}\times \rightarrow \mathbb{C} $ is a continuous function so that $\displaystyle\int_{a}^{x}K\left( x,\tau ,z\left( \tau \right) ,z\left( \delta \left( \tau \right) \right) \right) d\tau $ is a bounded continuous function. In addition, suppose that $K$ satisfies the Lipschitz condition
\begin{equation}\label{eq46}
\left\vert K\left( x,t,u,w\right) -K\left( x,t,v,z\right) \right\vert \leq L\left\vert w-z\right\vert 
\end{equation}
with $L>0$.

If $y\in C_{b}^{1}\left[ a,\infty \right) $ is such that
\begin{equation} \label{eq47}
\left\vert ^{H}\mathbb{D}_{a+}^{\alpha ,\beta ;\psi }y\left( x\right) -f\left( x,y\left( x\right) ,\int_{a}^{x}K\left( x,\tau ,y\left( \tau \right) ,y\left( \delta \left( \tau \right) \right) \right) d\tau \right) \right\vert \leq \sigma
\left( x\right) 
\end{equation}
$x\in \left[ a,\infty \right) $ and $M\left( \xi +\xi ^{2}L\right) <1$, then there is a unique function $y_{0}\in C_{b}^{1}\left[ a,b\right] $ such that 
\begin{equation}\label{eq48}
^{H}\mathbb{D}_{a+}^{\alpha ,\beta ;\psi }y_{0}\left( x\right) =f\left( x,y_{0}\left( x\right) ,\int_{a}^{x}K\left( x,\tau ,y_{0}\left( \tau \right) ,y_{0}\left( \delta \left( \tau \right) \right) \right) d\tau \right) 
\end{equation}
and
\begin{equation}\label{eq49}
\left\vert y\left( x\right) -y_{0}\left( x\right) \right\vert \leq \frac{\xi \sigma \left( x\right) }{1-M\left( \xi +\xi ^{2}L\right) }
\end{equation}
for all $x\in \left[ a,\infty \right) $.

This means that under the above conditions, the fractional integro-differential equation \textnormal{Eq.(\ref{eq43})} has the Ulam-Hyers-Rassias stability.
\end{teorema}

\begin{proof} For any $n\in \mathbb{N} $, we will define $I_{n}=\left[ a,a+n\right] $. By Theorem \ref{teo4}, there exists
a unique bounded differentiable function $y_{0,n}:I_{n}\rightarrow \mathbb{C}$ such that
\begin{equation*}
y_{0,n}\left( x\right) =\frac{\left( \psi \left( x\right) -\psi \left( a\right) \right) ^{\gamma -1}}{\Gamma \left( \gamma \right) } c+I_{a+}^{\alpha ;\psi }f\left( x,y_{0,n}\left( x\right) ,\int_{a}^{x}K\left( x,\tau ,y_{0,n}\left( \tau \right) ,y_{0,n}\left( \delta \left( \tau \right) \right) \right) d\tau \right) 
\end{equation*}
and
\begin{equation}\label{eq50}
\left\vert y\left( x\right) -y_{0,n}\left( x\right) \right\vert \leq \frac{\xi \sigma \left( x\right) }{1-M\left( \xi +\xi ^{2}L\right) }
\end{equation}
for all $x\in I_{n}.$ The uniqueness of $y_{0,n}$ implies that: if $x\in I_{n}$ then
\begin{equation}\label{eq51}
y_{0,n}\left( x\right) =y_{0,n+1}\left( x\right) =y_{0,n+2}\left( x\right) =\cdot \cdot \cdot
\end{equation}

For any $x\in \left[ a,\infty \right) $, let us define $n\left( x\right) =\min \left\{ n\in \mathbb{N} ;\text{ }x\in I_{n}\right\} $ and the function $y_{0}:\left[ a,\infty \right) \rightarrow 
\mathbb{C}$ by
\begin{equation}\label{eq52}
y_{0}\left( x\right) =y_{0,n\left( x\right) }\left( x\right).
\end{equation}

For any $x_{1}\in \left[ a,\infty \right) $, let $n_{1}=n\left( x_{1}\right) .$ Then $x_{1}\in $ Int $I_{n+1}$ there exists an $\varepsilon >0$ such that  $y_{0}\left( x\right) =y_{0,n+1}\left( x\right) $ for all $x\in \left(
x_{1}-\varepsilon ,\text{ }x_{1}+\varepsilon \right)$. By Theorem \ref{teo4} $y_{0,n+1}$ is continuous at $x_{1},$ and so it is $y_{0}$.

Now, we will prove that $y_{0}$ satisfies
\begin{equation} \label{eq53}
y_{0}\left( x\right) =\frac{\left( \psi \left( x\right) -\psi \left( a\right) \right) ^{\gamma -1}}{\Gamma \left( \gamma \right) }c+I_{a+}^{\alpha ;\psi }f\left( x,y_{0}\left( x\right) ,\int_{a}^{x}K\left(
x,\tau ,y_{0}\left( \tau \right) ,y_{0}\left( \delta \left( \tau \right) \right) \right) d\tau \right)
\end{equation}
and Eq.(\ref{eq49}). For an arbitrary $x\in \left[ a,\infty \right) $ we choose $n\left( x\right) $ such that $x\in I_{n\left( x\right)}$. By Eq.(\ref{eq50}) and Eq.(\ref{eq52}), we have
\begin{eqnarray}
y_{0}\left( x\right)  &=&y_{0,n\left( x\right) }\left( x\right) =\frac{%
\left( \psi \left( x\right) -\psi \left( a\right) \right) ^{\gamma -1}}{%
\Gamma \left( \gamma \right) }c  \notag  \label{eq54} \\
&&+I_{a+}^{\alpha ;\psi }f\left( x,y_{0,n\left( x\right) }\left( x\right)
,\int_{a}^{x}K\left( x,\tau ,y_{0,n\left( x\right) }\left( \tau \right)
,y_{0,n\left( x\right) }\left( \delta \left( \tau \right) \right) \right)
d\tau \right)   \notag \\
&=&\frac{\left( \psi \left( x\right) -\psi \left( a\right) \right) ^{\gamma
-1}}{\Gamma \left( \gamma \right) }c+I_{a+}^{\alpha ;\psi }f\left(
x,y_{0}\left( x\right) ,\int_{a}^{x}K\left( x,\tau ,y_{0}\left( \tau \right)
,y_{0}\left( \delta \left( \tau \right) \right) \right) d\tau \right) . 
\notag \\
&&
\end{eqnarray}

Note that $n\left( \tau \right) \leq n\left( x\right) $, for any $\tau \in I_{n\left( x\right) }$, and it follows from Eq.(\ref{eq51}) that $y_{0}\left( \tau \right) =y_{0,n\left( \tau \right) }\left( \tau \right) =y_{0,n\left(
x\right) }\left( \tau \right) $, so, the last equality in  Eq.(\ref{eq54}) holds true. To prove Eq.(\ref{eq49}), by Eq.(\ref{eq52}) and Eq.(\ref{eq50}), we have for all $x\in \left[ a,\infty \right) $
\begin{equation}\label{eq55}
\left\vert y\left( x\right) -y_{0}\left( x\right) \right\vert =\left\vert y\left( x\right) -y_{0,n\left( x\right) }\left( x\right) \right\vert \leq  \frac{\xi \sigma \left( x\right) }{1-M\left( \xi +L\xi ^{2}\right) }.
\end{equation}

Now, we will prove the uniqueness of $y_{0}$. Consider another bounded differentiable function $y_{1}$ which satisfies Eq.(\ref{eq48}) and Eq.(\ref{eq49}), for all $x\in \left[ a,\infty \right) $. By the uniqueness of the solution on $I_{n\left( x\right) }$ for any $n\left( x\right) \in \mathbb{N}$ we have $y_{0\mid _{I_{n\left( x\right) }}}=y_{0,n\left( x\right) }$ and $y_{1\mid _{I_{n\left( x\right) }}}$ satisfying Eq.(\ref{eq48}) and Eq.(\ref{eq49}) for all $x\in I_{n\left( x\right) }$ and so
\begin{equation}\label{eq56}
y_{0}\left( x\right) =y_{0\mid _{I_{n\left( x\right) }}}=y_{1\mid_{I_{n\left( x\right) }}}=y_{1}.
\end{equation}
\end{proof}

It is also possible to obtain proof of Theorem \ref{teo6}, for semi infinite interval of the type $\left( -\infty ,b\right]$ with $b\in \mathbb{R}$ and $\left( -\infty ,\infty \right)$, just choose the necessary conditions.

\section{Concluding remarks}
In this paper, we have shown that, by means of the $\psi$-Hilfer fractional derivative and the Banach fixed-point theorem, the class of fractional integro-differential equations allows the stability of Ulam-Hyers, Ulam-Hyers-Rassias and semi-Ulam-Hyers-Rassias in the intervals $[a,b]$ and $[a,\infty)$. The investigation of the Ulam-Hyers stability theme in the FC has been growing and other works related to stability with the $\psi$-Hilfer fractional derivative and other derivatives \cite{ZE3,ZE4} will be presented in the near future.

\bibliography{ref}
\bibliographystyle{plain}

\end{document}